\documentclass[a4paper,leqno]{amsart}

\usepackage{latexsym}
\usepackage[english]{babel}
\usepackage{fancyhdr}
\usepackage[mathscr]{eucal}
\usepackage{amsmath}
\usepackage{mathrsfs}
\usepackage{amsthm}
\usepackage{amsfonts}
\usepackage{amssymb}
\usepackage{amscd}
\usepackage{bbm}
\usepackage{graphicx}
\usepackage{graphics}
\usepackage{latexsym}
\usepackage{color}

\newcommand{\ud}{\mathrm{d}}

\newcommand{\ii}{\mathrm{i}}
\newcommand{\cH}{\mathcal{H}}

\theoremstyle{plain}
\newtheorem{theorem}{Theorem}[section]
\newtheorem{lemma}[theorem]{Lemma}
\newtheorem{corollary}[theorem]{Corollary}
\newtheorem{proposition}[theorem]{Proposition}

\theoremstyle{definition}

\newtheorem{remark}[theorem]{Remark}
\newtheorem*{remark*}{Remark}

\numberwithin{equation}{section}

\begin{document}

\title[On geometric quantum confinement in Grushin-type manifolds]
{On geometric quantum confinement in Grushin-type manifolds}
\author[M.~Gallone]{Matteo Gallone}
\address[M.~Gallone]{International School for Advanced Studies -- SISSA \\ via Bonomea 265 \\ 34136 Trieste (Italy).}
\email{mgallone@sissa.it}
\author[A.~Michelangeli]{Alessandro Michelangeli}
\address[A.~Michelangeli]{International School for Advanced Studies -- SISSA \\ via Bonomea 265 \\ 34136 Trieste (Italy).}
\email{alemiche@sissa.it}
\author[E.~Pozzoli]{Eugenio Pozzoli}
\address[E.~Pozzoli]{Sorbonne Universit\'e, Universit\'e Paris-Diderot SPC, CNRS, INRIA, 
Laboratoire Jacques-Louis Lions, \'equipe Cage, F-75005 Paris (France).}
\email{eugenio.pozzoli@inria.fr}


\begin{abstract}
We study the problem of so-called geometric quantum confinement in a class of two-dimensional incomplete Riemannian manifold with metric of Grushin type. We employ a constant-fibre direct integral scheme, in combination with Weyl's analysis in each fibre, thus fully characterising the regimes of presence and absence of essential self-adjointness of the associated Laplace-Beltrami operator. 
\end{abstract}

\date{\today}

\subjclass[2000]{}

%

\maketitle


\section{Introduction. Geometric quantum confinement.}
\label{sec:intro}

The notion of confinement for Schr\"{o}dinger's evolution refers, in a wide generality, to the feature that a solution to Schr\"{o}dinger's equation remains localised in an appropriate sense within a fixed spatial region, uniformly for all times. In the applications this is naturally referred to also as `quantum' confinement, a jargon that we are going to refine in a moment.

As well known, the precise features of quantum confinement are determined by the evolutive properties of the `free' Schr\"{o}dinger Hamiltonian in combination with the properties of the additional confining potential and of the underlying geometry of the space.

One of the most typical and relevant setting concerns a non-relativistic quantum particle moving on an orientable Riemannian manifold $(M,g)$ of dimension $d\in\mathbb{N}$ equipped with a smooth measure $\mu$, conventionally the measure $\mu_g:=\mathrm{vol}_g$ associated with the Riemannian volume form: the Hilbert space for the system is $L^2(M,\ud\mu)$ and the Schr\"{o}dinger Hamiltonian of interest is a self-adjoint realisation of the operator $H:=-\Delta_\mu+V$, where $\Delta_\mu=\mathrm{div}_\mu\circ\nabla$ is the Laplace-Beltrami operator computed with respect to the measure $\mu$ and $V$ is a real-valued potential on $M$.

In this setting the problem of quantum confinement is posed and interpreted as follows. The manifold $M\subset\mathbb{R}^d$ is chosen to be precisely the open spatial region which one wants to inquire whether the quantum particle remains confined in, and the operator $H$ is initially defined on the minimal domain $C^\infty_c(M)$, the dense subspace of $L^2(M,\ud\mu)$ of smooth functions with compact support. If such a choice does \emph{not} make $H$ essentially self-adjoint and hence leaves room for a multiplicity of distinct self-adjoint extensions of $H$, then the domain of each extension $\widetilde{H}$ is qualified by suitable boundary conditions of self-adjointness, and Schr\"{o}dinger's unitary flow $e^{-\ii t \widetilde{H}}$ evolves the quantum particle's wave-function so as to reach the boundary $\partial M$ in finite time, which is interpreted as a \emph{lack of confinement}. This is natural if one thinks of boundary conditions as describing a `physical interaction' of the boundary with the interior: the need for such an interaction, as a condition to make the Hamiltonian self-adjoint and hence to make the evolved wave function $e^{-\ii t \widetilde{H}}\psi_0$ belong to $L^2(M,\ud\mu)$ for all times for initial $\psi_0$ in the domain of $\widetilde{H}$, is the opposite of `confinement in $M$ without confining boundaries'. For this reason, if on the other hand $H$ is essentially self-adjoint on $C^\infty_c(M)$, then it is natural to say that the dynamics generated by its closure $\overline{H}$ exhibits quantum confinement in $M$: no quantum information escapes from $M$.

Thus, for example, the well-known fact that the ordinary Laplacian $H=-\frac{\ud^2}{\ud x^2}$ on the interval $M=(0,1)$ with the induced Euclidean metric, defined initially on $C^\infty_c(0,1)$, admits a four-real-parameter family of self-adjoint extensions in $L^2((0,1),\ud x)$, each of which is characterised by a linear relation between the values of the function and of its derivative at $x=0$ and the analogous values at $x=1$, is interpreted by saying that a quantum particle moving freely in the interval $(0,1)$  remains within that interval thanks to the appropriate boundary conditions, hence there is no `natural' quantum confinement. Pictorially, if an initial smooth function $\psi_0$ supported in $(0,1)$ was evolved instead as an element of the Hilbert space $L^2(\mathbb{R})$ subject to the dynamics generated by the closure of the essentially self-adjoint Laplacian $H=-\frac{\ud^2}{\ud x^2}$ with domain $C^\infty_c(\mathbb{R})$, then $e^{-\ii t\overline{H}}\psi_0$ would `exit' from the interval $(0,1)$ at later times $t>0$.

In short, the issue of the quantum confinement within the manifold $M$ is the issue of the essential self-adjointness of the operator $-\Delta_\mu+V$ with domain $C^\infty_c(M)$ in the Hilbert space $L^2(M,\ud\mu)$.

The case of smooth and \emph{geodesically complete} Riemannian manifolds is relatively well-understood: in this case the essential self-adjointness of $-\Delta_\mu$ is by now a classical result \cite{Gaffney-1954} and fairly general sufficient conditions on $V$ are known to ensure the essential self-adjointness of $-\Delta_\mu+V$ \cite{Braverman-Milatovich-Shubin-2002}. Most understood is the special, fundamental case $M=\mathbb{R}^d$: a whole industry was built in understanding the self-adjointness of Schr\"{o}dinger's operators on $d$-dimensional Euclidean space, with a vast and by now classical literature -- see, e.g., \cite[Chapter X]{rs2} or \cite[Chapter 1]{Cycon-F-K-S-Schroedinger_ops}.

For \emph{incomplete} Riemannian manifolds the picture is less developed, yet fairly general classes of $V$'s are known which ensure the self-adjointness of Schr\"{o}dinger's operators on bounded domains of $\mathbb{R}^d$ with smooth boundary of co-dimension 1 \cite{Nenciu-Nenciu-2009}, or more generally on bounded domains of $\mathbb{R}^d$ with non-empty boundary \cite{Ward-2016}. In such cases it is fundamental for the essential self-adjointness of $-\Delta+V$, where $\Delta$ is now the Euclidean Laplacian, that $V$ blows up at the boundary of the considered domain.

Recently, quantum confinement within incomplete Riemannian manifolds has attracted considerable attention especially when the measure has degeneracies or singularities near the metric boundary  \cite{Boscain-Laurent-2013,Prandi-Rizzi-Seri-2016,Franceschi-Prandi-Rizzi-2017}. Such setting is intimately related with that of manifolds equipped with so-called \emph{almost-Riemannian structure} \cite{Agrachev-Boscain-Sigalotti-2008}, a notion that informally speaking refers to a smooth $d$-dimensional  manifold $\mathcal{M}$ equipped with a family of smooth vector fields $X_1,...,X_d$ satisfying the Lie bracket generating condition: if $\mathcal{Z} \subset\mathcal{M}$ is the embedded hyper-surface of points where the $X_j$'s are not linearly independent, on $\mathcal{M}\setminus \mathcal{Z}$ the fields $X_1,\dots,X_d$ define a Riemannian structure which however becomes singular on $\mathcal{Z}$.
%
%
%
When, for concreteness, $\mathcal{Z}$ partitions $\mathcal{M}$ into two regions $M_1$ and $M_2$, so that $\mathcal{M}=M_1\cup\mathcal{Z}\cup M_2$ and $M:=M_1\cup M_2$ as disjoint unions, the essential self-adjointness of $H=-\Delta_{\mu_g}+V$ on $L^2(M,\ud\mu_g)$ with domain $C^\infty_c(M)$ is interpreted, from the perspective of quantum confinement, as the impossibility that a function initially supported only on $M_1$ evolves \emph{across} $\mathcal{Z}$ so as to become supported also in $M_2$ in the course of the Schr\"{o}dinger dynamics generated by $\overline{H}$.

For Schr\"{o}dinger operators on incomplete Riemannian manifolds with singular measure near the metric boundary, sufficient conditions of (essential) self-adjointness, including curvature-based criteria, have been recently established in \cite{Boscain-Laurent-2013,Prandi-Rizzi-Seri-2016} -- we are going to comment further on such results in due time. In this context, a special focus is given to the case $V\equiv 0$, that is, a quantum particle not subject to external interaction: in this case quantum confinement, when it occurs, is then \emph{purely geometric}.

Geometric quantum confinement on an incomplete Riemannian manifold $(M,g)$ is of particular interest from one further perspective, owing to the sharp difference between the corresponding classical and quantum motion. In the former, since geodesics represent the classical trajectories, the classical particle does reach the boundary $\partial M$ in \emph{finite time}, whereas in the latter, because of the essential self-adjointness of $-\Delta_g$ with domain $C^\infty_c(M)$, \emph{for all times} the quantum particle's wave-function need not be qualified by boundary conditions at $\partial M$ -- pictorially, the quantum particle stays permanently away from $\partial M$.

In view of the above discussion, we are now ready to present our work.

We are primarily concerned here with characterising the \emph{occurrence} as well as the \emph{absence} of geometric quantum confinement in a concrete class of two-dimensional incomplete Riemannian manifolds, the so-called \emph{Grushin-type} manifolds. The prototypical example is the half-plane $M=\mathbb{R}^+\times\mathbb{R}$ equipped with metric $g=\ud x\otimes\ud x+x^{-2}\ud y\otimes \ud y=\ud x^2+x^{-2}\ud y^2$, where $(x,y)\in \mathbb{R}^+\times\mathbb{R}$.

As we are going to explain in detail, the main features of our study, also with respect to the previous recent studies of the same or of analogous problems, are:
\begin{itemize}
 \item the novelty, and conciseness, of the approach, based on an analysis of constant-fibre direct integral on Hilbert space in combination with Weyl's limit-point limit-circle argument for each fibre;
 \item the formulation of \emph{necessary and sufficient} conditions for the essential self-adjointness of the considered Laplace-Beltrami operators -- the previous works \cite{Boscain-Laurent-2013,Prandi-Rizzi-Seri-2016} (see also \cite[Remark 1.1]{Franceschi-Prandi-Rizzi-2017}) 
 only provided sufficient conditions for quantum confinement, without characterising the absence of it;
 \item the possibility of covering the `non-compact' case, as in the above example $(M,g)=(\mathbb{R}^+\times\mathbb{R},\ud x^2+x^{-2}\ud y^2)$ (for which, as said, previously only sufficient conditions of quantum confinement were found \cite{Boscain-Laurent-2013,Prandi-Rizzi-Seri-2016,Franceschi-Prandi-Rizzi-2017}),  
 thus generalising the analysis of \cite[Sect.~4]{Boscain-Laurent-2013} and of \cite{Boscain-Prandi-JDE-2016} made for the counterpart case $M=\mathbb{R}^+_x\times\mathbbm{T}_y$, where the $y$-variable was compactified;
 \item some degree of `robustness' of our approach, as we can immediately export it to a fairly general class of Grushin-type manifolds, say, the half-plane $M=\mathbb{R}^+_x\times\mathbb{R}_y$ with metric $ g=\ud x^2+F(x)\ud y^2$ for suitable functions $F$, thus simplifying the effective potential approach of \cite{Prandi-Rizzi-Seri-2016} based on Agmon-type estimates.
\end{itemize}

Besides, we shall also recognise that in the absence of essential self-adjointness, the considered Laplace-Beltrami operator has \emph{infinite} deficiency index. This rises the question of classifying and studying the vast multiplicity of self-adjoint extensions in relation to the behaviour at the boundary, especially the physically relevant ones characterised by \emph{local} boundary conditions, interpreting each self-adjoint realisation as a different mechanism how the quantum particle tends to `cross' the boundary itself. We intend to treat such an analysis in a forthcoming follow-up work.

\section{Setting of the problem and main results}

We consider the family $\{M_\alpha\equiv(M,g_\alpha)\,|\,\alpha\in[0,+\infty)\}$ of Riemannian manifolds defined by
\begin{equation}\label{eq:Mgalpha}
 \begin{split}
 M\;&:=\;\{(x,y)\,|\,x\in\mathbb{R}^+,\,y\in\mathbb{R}\} \\
 g_\alpha\;&:=\;\ud x\otimes\ud x+\frac{1}{x^{2\alpha}}\,\ud y\otimes\ud y\;=\;\ud x^2+\frac{1}{x^{2\alpha}}\,\ud y^2\,.
 \end{split}
\end{equation}
The value $\alpha=1$ selects the standard example of two-dimensional \emph{Grushin manifold} \cite[Chapter 11]{Calin-Chang-SubRiemannianGeometry}, or \emph{Grushin plane}, and all other members of the above family, as well as of the even larger family defined in \eqref{eq:Mf} below, will be generically referred to as (two-dimensional) \emph{Grushin-type manifolds}. The value $\alpha=0$ selects the Euclidean half-plane.

A straightforward computation \cite{Agrachev-Boscain-Sigalotti-2008,Boscain-Laurent-2013,Pozzoli_MSc2018} shows that the Gaussian (sectional) curvature $K_\alpha$ of $M_\alpha$ is
\begin{equation}
 K_\alpha(x,y)\;=\;-\frac{\,\alpha(\alpha+1)\,}{x^2}\,,
\end{equation}
hence $M_\alpha$ is a hyperbolic manifold whenever $\alpha>0$.

Each $M_\alpha$ is clearly parallelizable, a global orthonormal frame being
\begin{equation}\label{eq:frame}
\{X_1,X_2^{(\alpha)}\}\;=\;\left\{ 
\begin{pmatrix}
 1 \\ 0
\end{pmatrix},\;
\begin{pmatrix}
 0 \\ x^{\alpha}
\end{pmatrix}
\right\}\equiv\;\Big\{\frac{\partial}{\partial x},x^{\alpha}\frac{\partial}{\partial y}\Big\}.
\end{equation}

\begin{remark}\label{rem:ARS-noARS}
Upon extending $\{X_1,X_2^{(\alpha)}\}$ to the whole $\mathbb{R}^2$ with $X_2^{(\alpha)}:=\begin{pmatrix}
 0 \\ |x|^{\alpha}
\end{pmatrix}$
%
%
and defining
\begin{equation}
 \begin{split}
  M^+\;&:=\;M\,,\qquad M^-\;:=\{(x,y)\,|\,x\in\mathbb{R}^-,\,y\in\mathbb{R}\}\,, \\
  \mathcal{Z}\;&:=\;\{(0,y)\,|\,y\in\mathbb{R}\}\,,
 \end{split}
\end{equation}
one has now the Lie bracket $[X_1,X_2^{(\alpha)}]=\begin{pmatrix}
 0 \\ \alpha|x|^{\alpha-1}
\end{pmatrix}$.
Thus, if $\alpha=1$ the fields $X_1,X_2^{(\alpha)}$ define an \emph{almost-Riemannian structure} on $\mathbb{R}^2=M^+\cup\mathcal{Z}\cup M^-$, following the notation used in the Introduction, for a rigorous definition of which we refer, e.g., to \cite[Sec.~1]{Agrachev-Boscain-Sigalotti-2008} or \cite[Sect.~7.1]{Prandi-Rizzi-Seri-2016}: indeed the Lie bracket generating condition 
\begin{equation}
 \dim\mathrm{Lie}_{(x,y)}\,\mathrm{span}\{X_1,X_2^{(\alpha)}\}\;=\;2\qquad\forall(x,y)\in\mathbb{R}^2\,,
\end{equation}
is satisfied in this case. For $\alpha\in(0,1)$ the field $X_2^{(\alpha)}$ is \emph{not} smooth, which prevents $X_1,X_2^{(\alpha)}$ to define an almost-Riemannian structure. However, on $\mathbb{R}^2\!\setminus\!\mathcal{Z}$ the fields $X_1,X_2^{(\alpha)}$ do define a Riemannian structure for every $\alpha\geqslant 0$ given by
\begin{equation}\label{eq:galphaeverywhere}
  g_\alpha\;:=\;\ud x\otimes\ud x+\frac{1}{|x|^{2\alpha}}\,\ud y\otimes\ud y\,.
\end{equation}
\end{remark}


To each $M_\alpha$ one naturally associates the Riemannian volume form
\begin{equation}\label{eq:volumeform}
 \mu_\alpha\;:=\;\mathrm{vol}_{g_\alpha}\;=\;\sqrt{\det g_\alpha}\,\ud x\wedge\ud y\;=\;x^{-\alpha}\,\ud x\wedge\ud y\,.
\end{equation}
By means of \eqref{eq:frame} and \eqref{eq:volumeform} one computes
\begin{equation}
 \begin{split}
  X_1^2\;&=\;\frac{\partial^2}{\partial x^2}\,,\qquad \mathrm{div}_{\mu_\alpha}X_1\;=\;-\frac{\alpha}{x}\,, \\
  (X_2^{(\alpha)})^2\;&=\;x^{2\alpha}\frac{\partial^2}{\partial y^2}\,,\quad \mathrm{div}_{\mu_\alpha}X_2^{(\alpha)}\;=\;0\,,
 \end{split}
\end{equation}
whence
\begin{equation}\label{eq:Deltamualpha}
\begin{split}
 \Delta_{\mu_\alpha}\;&=\;\mathrm{div}_{\mu_\alpha}\circ\nabla \\
 &=\;X_1^2+X_2^2+(\mathrm{div}_{\mu_\alpha}X_1)X_1+(\mathrm{div}_{\mu_\alpha}X_2^{(\alpha)})X_2^{(\alpha)} \\
 &=\;\frac{\partial^2}{\partial x^2}+x^{2\alpha}\frac{\partial^2}{\partial y^2}-\frac{\alpha}{x}\,\frac{\partial}{\partial x}\,,
\end{split}
\end{equation}
which is the (Riemannian) Laplace-Beltrami operator on $M_\alpha$.

Before entering the core of our analysis let us establish a preliminary property that is folk knowledge to some extent: we state it here for the benefit of the reader.

\begin{theorem}\label{thm:incompleteness}
 Let $\alpha>0$. All geodesics passing through a generic point $(x_0,y_0)\in M$ escape from $M_\alpha$.
\end{theorem}

Theorem \ref{thm:incompleteness} will be proved and discussed in Section \ref{sec:incompleteness}, and for $\alpha=1$ it can be found already, e.g., in \cite[Sect.~11.2]{Calin-Chang-SubRiemannianGeometry} or \cite[Sect.~3.1]{Boscain-Laurent-2013}: it clearly implies the geodesic incompleteness of all the $M_\alpha$'s, a feature that is already evident by observing that $y=y_0$ is a geodesic line.

Next, in the Hilbert space
\begin{equation}
 \cH_\alpha\;:=\;L^2(M,\ud\mu_\alpha)\,,
\end{equation}
understood as the completion of $C^\infty_c(M)$ with respect to the scalar product
\begin{equation}
 \langle \psi,\varphi\rangle_{\alpha}\;:=\;\iint_{\mathbb{R}^+\times\mathbb{R}}\overline{\psi(x,y)}\,\varphi(x,y)\,\frac{1}{x^{\alpha}}\,\ud x\,\ud y\,,
\end{equation}
we consider the `\emph{minimal}' free Hamiltonian
\begin{equation}\label{Halpha}
 H_\alpha\;:=\;-\Delta_{\mu_\alpha}\,,\qquad\mathcal{D}(H_\alpha)\;:=\;C^\infty_c(M)\,,
\end{equation}
which is a densely defined, symmetric, lower semi-bounded operator (symmetry in particular follows from Green's identity).

Our main question then becomes for which $\alpha$'s the operator $H_\alpha$ is or is not essentially self-adjoint with respect to the Hilbert space $\cH_\alpha$, and hence for which $\alpha$'s one has or has not purely geometric quantum confinement in the manifold $M_\alpha$.

As mentioned already, the study of this problem has precursors in the literature. The essential self-adjointness of $H_\alpha$ for $\alpha\geqslant 1$ is proved with several related approaches in 
\cite{Boscain-Laurent-2013,Prandi-Rizzi-Seri-2016,Franceschi-Prandi-Rizzi-2017}. In particular, \cite{Boscain-Laurent-2013} is eminently perturbative in nature, whereas the completeness criterion \cite[Theorem 3.1]{Prandi-Rizzi-Seri-2016} exploits an approach of `effective potential', an intrinsic function depending only on the Riemannian structure of the manifold, which in the present case amounts to
\begin{equation}\label{eq:Veff}
 V_\mathrm{eff}(x,y)\;=\;\frac{\,\alpha(\alpha+2)\,}{4x^2}\,.
\end{equation}
When $\alpha\geqslant 1$, by means of \eqref{eq:Veff} and Hardy's inequality it is possible to express the lower-semiboundedness of the quadratic form of $H_\alpha$ as an Agmon-type estimate, which in turn allows one to deduce that the eigenfunction problem $H_\alpha^*\psi=E\psi$ for sufficiently negative $E$ can be only solved by $\psi=0$, a typical signature of essential self-adjointness for $H_\alpha$. 
In this respect, \cite[Theorem 3.1]{Prandi-Rizzi-Seri-2016} does not exclude that in the regime $\alpha\in[0,1)$ the essential self-adjointness could still hold.

From a closely related perspective, we also mention the analysis of \cite[Sect.~3.2]{Boscain-Laurent-2013} and of \cite{Boscain-Prandi-JDE-2016} on the quantum confinement problem for a \emph{compactified} version of $M_\alpha$, the manifold
\begin{equation}\label{eq:MtildeBoscain}
 \widetilde{M}_\alpha\:\equiv\:(\widetilde{M},g_\alpha)\qquad\textrm{with}\qquad \widetilde{M}\;:=\;\{(x,y)\,|\,x\in\mathbb{R}^+,\,y\in\mathbb{T}\}\,.
\end{equation}
In this case it is possible to exploit the compactness of the torus $\mathbb{T}$ in such a way to pass, through a Fourier transform in $y$, to a setting of infinite-orthogonal-sum Hilbert space, which allows one to qualify the presence or absence of essential self-adjointness of the associated Laplace-Beltrami in terms of an auxiliary problem on the half-line space $L^2(\mathbb{R}^+\,\ud x)$, and for the latter the classical limit-point/limit-circle analysis of Weyl does the job. As our approach in practice generalises this idea from infinite orthogonal sums to constant-fibre direct integrals, so as to deal with a \emph{non-compact} $y$-variable, we shall comment further on this point in Section \ref{sec:additional_remarks} below.

We characterise the essential self-adjointness of the operator \eqref{Halpha} by means of an alternative method that allows us to solve the problem for all $\alpha$'s, with no need to simplify it with a compactified version of the Grushin plane, and in a way that to our taste clarifies the \emph{operator-theoretic mechanism} for self-adjointness. Our main results read as follows.

\begin{theorem}\label{thm:no_confinement}
 If $\alpha\in[0,1)$, then the operator $H_\alpha$ is not essentially self-adjoint and therefore there is no geometric quantum confinement in the Grushin plane $M_\alpha$.
\end{theorem}

\begin{theorem}\label{thm:confinement}
 If $\alpha\in[1,+\infty)$, then the operator $H_\alpha$ is essentially self-adjoint and therefore the Grushin plane $M_\alpha$ provides geometric quantum confinement.
\end{theorem}

\begin{remark}
 The presence of geometric quantum confinement can be re-interpreted as follows (see also the discussion in \cite[Sect.~4.1]{Boscain-Laurent-2013}). Clearly,
 \begin{equation}\label{eq:decomp+-}
  L^2(M^+\cup M^-,\ud\mu_\alpha)\;\cong\;L^2(M^+,\ud \mu_\alpha)\oplus L^2(M^-,\ud \mu_\alpha)
 \end{equation}
 with $\mu_\alpha:=\mathrm{vol}_{g_\alpha}$ and $g_\alpha$ given by \eqref{eq:galphaeverywhere}, and if we set $H_\alpha^+:=H_\alpha$ and in complete analogy to \eqref{Halpha} we define $H_\alpha^-$ and $H_\alpha^{\cup}$ as densely defined, symmetric, lower semi-bounded operators respectively in $L^2(M^-,\ud \mu_\alpha)$ and  $L^2(M^+\cup M^-,\ud\mu_\alpha)$, by repeating verbatim the proof of Theorem \ref{thm:confinement} we see that $H_\alpha^+$, $H_\alpha^-$, and $H_\alpha^{\cup}$ are essentially self-adjoint in the respective Hilbert spaces whenever $\alpha\in[1,+\infty)$. It is also immediate to see that with respect to the decomposition \eqref{eq:decomp+-} one has $H_\alpha^{\cup}=H_\alpha^+\oplus H_\alpha^-$ and that 
 \begin{equation}
  \overline{H_\alpha^{\cup}}\;=\;\overline{H_\alpha^+}\oplus\overline{H_\alpha^-}\,.
 \end{equation}
 As a consequence, the propagators satisfy
 \begin{equation}
  e^{-\ii t \overline{H_\alpha^{\cup}}}\;=\;e^{-\ii t \overline{H_\alpha^+}}\oplus e^{-\ii t \overline{H_\alpha^-}}\,,\qquad \forall\,t\in\mathbb{R}\,.
 \end{equation}
 Therefore, for any initial datum $\psi_0\in\mathcal{D}(\overline{H_\alpha^{\cup}})\subset L^2(M^+\cup M^-,\ud\mu_\alpha)$ with support only within $M^+$, the \emph{unique} solution $\psi\in C^1(\mathbb{R}_t,L^2(M^+\cup M^-,\ud\mu_\alpha))$ to the Cauchy problem
 \begin{equation}
  \begin{cases}
   \;\ii\partial_t \psi \!\!&=\;\overline{H_\alpha^{\cup}}\,\psi \\
   \;\psi|_{t=0}\!\!&=\;\psi_0
  \end{cases}
 \end{equation}
 remains for all times supported (`confined') in $M^+$. The quantum particle initially prepared in the right open half-plane never crosses the $y$-axis towards the left half-plane.
\end{remark}

\begin{remark}
 In the absence of essential self-adjointness, the deficiency index of $H_\alpha$ is \emph{infinite}, as we shall show in the more general Theorem \ref{thm:generalisation_to_f}(iii) below. This opens the interesting problem, from the point of view of the quantum-mechanical interpretation, of classifying the self-adjoint extensions of $H_\alpha$ in terms of boundary conditions at the axis $x=0$, each generating a different dynamics in which the quantum particle `crosses the boundary'. In such an enormous family of extensions it is of interest, in particular, to discuss those qualified by `local' boundary conditions, the physically most natural ones. It is not difficult to show, and we intend to discuss these aspects in a follow-up analysis, that the Friedrichs extension satisfies Dirichlet boundary conditions and hence is the distinguished extension that preserves the confinement of the particle. All other extensions drive the particle up to the boundary. 
\end{remark}

\begin{remark}
 The lack of geometric quantum confinement in $M_\alpha$ for $\alpha\in[0,1)$ is compatible with the quantum confinement in \emph{regular} almost-Riemannian structures proved recently in \cite[Theorem 7.1]{Prandi-Rizzi-Seri-2016}. Indeed, as observed already in Remark \ref{rem:ARS-noARS}, what fails to hold in the first place is the almost-Riemannian structure on $\mathbb{R}^2$ with metric $g_\alpha$, owing to the non-smoothness of the field $X_2^{(\alpha)}$ in this regime of $\alpha$.
\end{remark}

As is going to emerge in the course of the proofs, our approach has a two-fold feature. On the one hand it is relatively `rigid', for it does not have an immediate generalisation in application to \emph{generic} almost-Riemannian structures, for which the more versatile, typically perturbative analyses of \cite{Boscain-Laurent-2013,Prandi-Rizzi-Seri-2016,Franceschi-Prandi-Rizzi-2017} appear as more efficient and informative. On the other hand, it is particularly `robust', whenever the problem can be boiled down to a constant-fiber direct integral scheme and to the study of self-adjointness along each fibre, and this allows us to cover a larger class of Grushin planes than that considered so far.

To this aim, let us introduce the manifold $M_f\equiv(M,g_f)$ by replacing \eqref{eq:Mgalpha} with
\begin{equation}\label{eq:Mf}
 g_f\;:=\;\ud x\otimes\ud x+f^2(x)\,\ud y\otimes\ud y
\end{equation}
for some measurable function $f$ on $\mathbb{R}$ satisfying
\begin{equation}\label{eq:assumtpions_f}
 \begin{array}{ll} 
  \mathrm{(i)} &   f(x)>0\;\;\;\forall x\neq 0 \\
  \mathrm{(ii)} &  f(x)\geqslant \kappa \;\;\;\textrm{in a neighbourhood of $x=0$ for some $\kappa>0$} \\
  \mathrm{(iii)} &  \textrm{$f\in C^\infty(\mathbb{R}\!\setminus\!\{0\})$} \\
  \mathrm{(iv)} & 2f(x)f''(x)-f'(x)^2\geqslant 0 \;\;\;\forall x\neq 0\,.
 \end{array}
\end{equation}
The interest in assumptions \eqref{eq:assumtpions_f} is precisely when $f$ becomes singular as $x\to 0$.
The reason of condition (iv) will be clarified in due time. The special choice considered above was $f(x)=|x|^{-\alpha}$: in this case condition (iv) reads $\alpha(2+\alpha) |x|^{-2(1+\alpha)}\geqslant 0$.  The smoothness in condition (iii) is required to match the definition of Riemannian manifold, otherwise we shall only use $C^2$-regularity.

(It is worth mentioning that this point of view, with the more general manifold $M_f$, is the same as that of \cite{Boscain-Laurent-2013}, and so are formulas \eqref{eq:frame-f}-\eqref{eq:unaltra} below: here in addition we take care of the explicit assumptions \eqref{eq:assumtpions_f} required on $f$, which finally allow us to prove our general Theorem \ref{thm:generalisation_to_f}.)

This yields a generalised Grushin plane with global orthonormal frame 
\begin{equation}\label{eq:frame-f}
\{X_1,X_2^{(f)}\}\;=\;\left\{ 
\begin{pmatrix}
 1 \\ 0
\end{pmatrix},\;
\begin{pmatrix}
 0 \\ 1/f(x)
\end{pmatrix}
\right\}\equiv\;\Big\{\frac{\partial}{\partial x},\frac{1}{f(x)}\frac{\partial}{\partial y}\Big\}\,,
\end{equation}
and a computation analogous to \eqref{eq:volumeform}-\eqref{eq:Deltamualpha} shows that the associated Laplace-Beltrami operator $\Delta_f\equiv\Delta_{\mu_{g_f}}$ ($\mu_{g_f}\equiv\mathrm{vol}_{g_f}=f(x)\,\ud x\wedge\ud y$) is given by
\begin{equation}\label{eq:unaltra}
 \Delta_{f}\;=\;\frac{\partial^2}{\partial x^2}+\frac{1}{f^2(x)}\,\frac{\partial^2}{\partial y^2}+\frac{f'(x)}{f(x)}\,\frac{\partial}{\partial x}\,.
\end{equation}
Let us then define the `\emph{minimal}' free Hamiltonian
\begin{equation}\label{Hf}
 H_f\;:=\;-\Delta_{f}\,,\qquad\mathcal{D}(H_f)\;:=\;C^\infty_c(M)\,,
\end{equation}
a densely defined, symmetric, lower semi-bounded operator in $\cH_f:=L^2(M,\ud\mu_{g_f})$.
The same scheme used for Theorems \ref{thm:no_confinement} and \ref{thm:confinement} allows us to discuss the essential self-adjointness of $H_f$. The result is the following.

\begin{theorem}\label{thm:generalisation_to_f}
 Let $f$ be a measurable function satisfying assumptions \eqref{eq:assumtpions_f} and let $H_f$ be the corresponding operator defined in \eqref{Hf}. 
 \begin{itemize}
  \item[(i)] If, point-wise for every $x\neq 0$, 
  \begin{equation}\label{eq:fcondition_conf}
   2 f f''-f'^2\;\geqslant\;\frac{3}{x^2}\,f^2\,,
  \end{equation}
  then $H_f$ is essentially self-adjoint with respect to the Hilbert space $\cH_f$, and therefore the generalised Grushin plane $M_f$ produces geometric quantum confinement.
  \item[(ii)] If, point-wise for every $x\neq 0$, 
  \begin{equation}\label{eq:fcondition_noconf}
   2 f f''-f'^2\;\leqslant\;\frac{3-\varepsilon}{x^2}\,f^2\qquad\textrm{for some }\varepsilon>0\,,
  \end{equation}
  then $H_f$ is not essentially self-adjoint, and therefore there is no geometric quantum confinement within the generalised Grushin plane $M_f$. 
  \item[(iii)] In case (ii) the operator $H_f$ has \emph{infinite} deficiency index.
 \end{itemize}
 \end{theorem}

 \begin{remark}
  Theorem \ref{thm:generalisation_to_f} reproduces Theorems \ref{thm:no_confinement} and \ref{thm:confinement} when one makes the special choice $f(x)=x^{-\alpha}$, for in this case
  \begin{equation*}\label{f_ratio_with_alpha}
   2 f f''-f'^2-\frac{3}{x^2}\,f^2\;=\;\frac{\,(\alpha-1)(3+\alpha)\,}{4x^2}\,,
  \end{equation*}
 whence the threshold value $\alpha=1$ between absence and presence of confinement. Conditions \eqref{eq:fcondition_conf}-\eqref{eq:fcondition_noconf} are homogeneous in $f$, thus the same conclusion holds for $f(x)=\lambda x^{-\alpha}$, $\lambda>0$: this amounts to dilate the $y$-axis, in practice leaving the metric unchanged. 
 \end{remark}

 Theorems \ref{thm:no_confinement}, \ref{thm:confinement}, and \ref{thm:generalisation_to_f} are going to be proved in Section \ref{sec:proofs} after an amount of preparation in Section \ref{sec:preliminaries}.


%
%
%
%
%
%
%
%
%

\section{Technical preliminaries}\label{sec:preliminaries}

\subsection{Unitarily equivalent reformulation}\label{sec:unitarily_equiv}~

Let us discuss the more general setting of Theorem \ref{thm:generalisation_to_f}, that is, the problem of the essential self-adjointness of the minimally defined Laplace-Beltrami operator \eqref{Hf} in the Hilbert space $L^2(M,\ud\mu_{g_f})=L^2(\mathbb{R}^+\times\mathbb{R},f(x)\ud x\ud y)$.

Through the unitary transformation
\begin{equation}
 U_f:L^2(\mathbb{R}^+\times\mathbb{R},f(x)\ud x\ud y)\stackrel{\cong}{\longrightarrow}L^2(\mathbb{R}^+\times\mathbb{R},\ud x\ud y)\,,\qquad \psi\mapsto f^{1/2}\psi
\end{equation}
a simple computation shows that
\begin{equation}\label{eq:diff_OP}
 \begin{split}
 U_fH_fU_f^{-1}\;&=\;-\frac{\partial^2}{\partial x^2}-\frac{1}{f^2}\frac{\partial^2}{\partial y^2}+\frac{2ff''-f'^2}{4f^2} \\
 \mathcal{D}(U_fH_fU_f^{-1})\;&=\;C^\infty_c(\mathbb{R}^+_x\times\mathbb{R}_y)\,.
 \end{split}
\end{equation}

The further unitary $\mathcal{F}_2:L^2(\mathbb{R}^+\times\mathbb{R},\ud x\ud y)\stackrel{\cong}{\longrightarrow}L^2(\mathbb{R}^+\times\mathbb{R},\ud x\ud \xi)$ consisting of the Fourier transform in the $y$-variable only produces the operator
\begin{equation}
 \mathscr{H}_f\;:=\;\mathcal{F}_2U_fH_fU_f^{-1}\mathcal{F}_2^{-1}
\end{equation}
whose domain and action are given by
\begin{equation}\label{eq:Hscrf}
 \begin{split}
  \mathscr{H}_f\;&=\;-\frac{\partial^2}{\partial x^2}+\frac{\xi^2}{f^2}+\frac{2ff''-f'^2}{4f^2} \\
  \mathcal{D}(\mathscr{H}_f)\;&=\;\{\psi\in L^2(\mathbb{R}^+\times\mathbb{R},\ud x\ud \xi)\,|\,\psi\in\mathcal{F}_2C^\infty_c(\mathbb{R}^+_x\times\mathbb{R}_y)\}\,.
 \end{split}
\end{equation}
Thus, for each $\psi\in\mathcal{D}(\mathscr{H}_f)$ the functions $\psi(\cdot,\xi)$ are compactly supported in $x$ inside $(0,+\infty)$ for every $\xi$, whereas the functions $\psi(x,\cdot)$ are some special case of Schwartz functions for every $x$.

The particular class of choices $f(x)=x^{-\alpha}$, $\alpha>0$, yield the operator
\begin{equation}
  \begin{split}
  \mathscr{H}_\alpha\;&=\;-\frac{\partial^2}{\partial x^2}+\xi^2 x^{2\alpha}+\frac{\,\alpha(2+\alpha)\,}{4x^2} \\
  \mathcal{D}(\mathscr{H}_\alpha)\;&=\;\{\psi\in L^2(\mathbb{R}^+\times\mathbb{R},\ud x\ud \xi)\,|\,\psi=\mathcal{F}_2C^\infty_c(\mathbb{R}^+_x\times\mathbb{R}_y)\}\,.
 \end{split}
\end{equation}

The self-adjointness problem for $H_f$, resp.~$H_\alpha$, is tantamount as the self-adjointness problem for $\mathscr{H}_f$, resp.~$\mathscr{H}_\alpha$, and it is this second problem that we are going to discuss.

\begin{remark}
 The `potential' (multiplicative) part of $\mathscr{H}_\alpha$, that is, $\frac{\,\alpha(2+\alpha)\,}{4x^2}$, is precisely the effective potential $V_\mathrm{eff}$ introduced in \cite{Prandi-Rizzi-Seri-2016} for the study of geometric confinement, computed for the special case of Grushin planes -- see \eqref{eq:Veff} above. Whereas in \cite{Prandi-Rizzi-Seri-2016} the intrinsic geometric nature of $V_\mathrm{eff}$ was emphasized, we can here supplement that interpretation by observing that $V_\mathrm{eff}$ encodes precisely the multiplicative contribution of the original Laplace-Beltrami operator when one transforms unitarily the underlying Hilbert space $L^2(M,\ud\mu_g)$, the unitary transformation being $\mathcal{F}_2\circ U_\alpha$. 
\end{remark}

For later purposes, let us also mention the following.

\begin{lemma}\label{lemma:Hfstar}
 The adjoint of $\mathscr{H}_f$ is the operator
\begin{equation}\label{eq:Hfstar}
 \begin{split}
  \mathscr{H}_f^*\;&=\;-\frac{\partial^2}{\partial x^2}+\frac{\xi^2}{f^2}+\frac{2ff''-f'^2}{4f^2} \\
  \mathcal{D}(\mathscr{H}_f^*)\;&=\;
  \left\{\!\!
  \begin{array}{c}
   \psi\in L^2(\mathbb{R}^+\times\mathbb{R},\ud x\ud \xi)\;\;\textrm{such that} \\
   \big(-\frac{\partial^2}{\partial x^2}+\frac{\xi^2}{f^2}+\frac{2ff''-f'^2}{4f^2}\big)\psi\in L^2(\mathbb{R}^+\times\mathbb{R},\ud x\ud \xi)
  \end{array}
  \!\!\right\}.
 \end{split}
\end{equation}
\end{lemma}

\begin{proof}
 $\mathscr{H}_f$ is unitarily equivalent, via Fourier transform in the second variable, to the minimally defined differential operator \eqref{eq:diff_OP}, whose adjoint is by standard arguments \cite[Sect.~1.3.2]{schmu_unbdd_sa} the maximally defined realisation of the same differential action, thus with domain consisting of the elements $F$'s such that both $F$ and $(-\frac{\partial^2}{\partial x^2}-\frac{1}{f^2}\frac{\partial^2}{\partial y^2}+\frac{2ff''-f'^2}{4f^2})F$ belong to $L^2(\mathbb{R}^+\times\mathbb{R},\ud x\ud y)$. Fourier-transforming such adjoint then yields \eqref{eq:Hfstar}. 
\end{proof}

\subsection{Constant-fibre direct integral scheme}~

Whereas obviously $L^2(\mathbb{R}^+_x\times\mathbb{R}_\xi,\ud x\ud \xi)\cong L^2(\mathbb{R}^+,\ud x)\otimes  L^2(\mathbb{R},\ud\xi)$, the operator $\mathscr{H}_f$ is not a simple product with respect to the above factorisation, it rather reads as the sum of two products
\begin{equation}
 \mathscr{H}_f\;=\;\Big(-\frac{\partial^2}{\partial x^2}+\frac{2ff''-f'^2}{4f^2}\Big)\otimes\mathbbm{1}_\xi+\frac{1}{f^2}\otimes \xi^2
\end{equation}
each of which with the same domain as $\mathscr{H}_f$ itself. The second summand is manifestly essentially self-adjoint on $L^2(\mathbb{R}^+,\ud x)\otimes  L^2(\mathbb{R},\ud\xi)$, whereas the self-adjointness of the first summand boils down to the analysis of the factor acting on $L^2(\mathbb{R}^+,\ud x)$ only, yet there is no general guarantee that the sum of the two preserves the essential self-adjointness.

It is more natural to regard $\mathscr{H}_f$ with respect to the constant-fibre direct integral structure
\begin{equation}\label{L^2integralDecomp}
 \begin{split}
  \cH\;:=\;L^2(\mathbb{R}^+\times\mathbb{R},\ud x\ud \xi)\;&\cong\;L^2\big(\mathbb{R},\ud \xi\,;L^2(\mathbb{R}^+,\ud x)\big) \\
  &\equiv\;\int_{\mathbb{R}}^{\oplus}\ud \xi \,L^2(\mathbb{R}^+,\ud x)\,,
 \end{split}
\end{equation}
thus thinking of $L^2(\mathbb{R}^+_x\times\mathbb{R}_\xi,\ud x\ud \xi)$ as $L^2(\mathbb{R}^+,\ud x)$-valued square-integrable functions of $\xi\in\mathbb{R}$. The space $\mathfrak{h}:=L^2(\mathbb{R}^+,\ud x)$ is the (constant) \emph{fibre} of the direct integral and the scalar products satisfy
\begin{equation}\label{eq:scalar_products_fibred}
 \langle \psi,\varphi\rangle_{\cH}\;=\;\int_{\mathbb{R}}\langle \psi(\cdot,\xi),\varphi(\cdot,\xi)\rangle_{\mathfrak{h}}\,\ud \xi\,.
\end{equation}
As well known, this is the natural scheme for the multiplication operator form of the spectral theorem \cite[Sect.~7.3]{Hall-2013_QuantumTheoryMathematicians}, as well as for the analysis of Schr\"{o}dinger's operators with periodic potentials \cite[Sect.~XIII.16]{rs4}, and we shall exploit this scheme here for the self-adjointness problem of $\mathscr{H}_f$.

For each $\xi\in\mathbb{R}$ we introduce the operator
\begin{equation}
 A_f(\xi)\;:=\;-\frac{\ud^2}{\ud x^2}+\frac{\xi^2}{f^2}+\frac{2ff''-f'^2}{4f^2}\,,\qquad\mathcal{D}(A_f(\xi))\;:=\;C^\infty_c(\mathbb{R}^+)
\end{equation}
acting on the fibre Hilbert space $\mathfrak{h}$. When $f(x)=x^{-\alpha}$ we write
\begin{equation}
 A_\alpha(\xi)\;:=\;-\frac{\ud^2}{\ud x^2}+\xi^2 x^{2\alpha}+\frac{\,\alpha(2+\alpha)\,}{4x^2}\,,\quad\,\mathcal{D}(A_\alpha(\xi))\;:=\;C^\infty_c(\mathbb{R}^+)\,.
\end{equation}
By construction the map $\mathbb{R}\ni\xi\mapsto  A_f(\xi)$ has values in the space of densely defined, symmetric operators on $\mathfrak{h}$, in fact all with the \emph{same} domain irrespectively of $\xi$, and all positive because of the assumptions on $f$. In each $A_f(\xi)$ $\xi$ plays the role of a fixed parameter. Moreover, all the $A_f(\xi)$'s are closable and each $\overline{A_f(\xi)}$ is positive and with the same dense domain in $\mathfrak{h}$. Arguing as for Lemma \ref{lemma:Hfstar} one has
\begin{equation}\label{eq:Afstar}
 \begin{split}
  A_f(\xi)^*\;&=\;-\frac{\ud^2}{\ud x^2}+\frac{\xi^2}{f^2}+\frac{2ff''-f'^2}{4f^2} \\
  \mathcal{D}(A_f(\xi)^*)\;&=\;
  \left\{\!\!
  \begin{array}{c}
   \psi\in L^2(\mathbb{R}^+,\ud x)\;\;\textrm{such that} \\
   \big(-\frac{\ud^2}{\ud x^2}+\frac{\xi^2}{f^2}+\frac{2ff''-f'^2}{4f^2}\big)\psi\in L^2(\mathbb{R}^+,\ud x)
  \end{array}
  \!\!\right\}.
 \end{split}
\end{equation}

Next, with respect to the decomposition \eqref{L^2integralDecomp} we define the operator $B_f$ in the Hilbert space $\cH$ by
\begin{equation}\label{ed:defB}
 \begin{split}
  \mathcal{D}(B_f)\;&:=\;\left\{\psi\in\cH\,\left|\! 
  \begin{array}{l}
   \mathrm{(i)}\quad\psi(\cdot,\xi)\in\mathcal{D}(\overline{A_f(\xi)})\textrm{ for almost every }\xi \\
   \mathrm{(ii)}\,\displaystyle\int_\mathcal{\mathbb{R}}\big\|\overline{A_f(\xi)}\psi(\cdot,\xi)\big\|_{\mathfrak{h}}^2\,\ud\xi<+\infty
  \end{array}
  \!\!\right.\right\}  \\
  (B_f\psi)(x,\xi)\;&:=\;\big(\overline{A_f(\xi)}\,\psi(\cdot,\xi)\big)(x)\,.
 \end{split}
\end{equation}
As customary, for the \emph{whole} \eqref{ed:defB} we use the symbol
\begin{equation}\label{eq:Bfdecomposable}
 B_f\;=\;\int_{\mathbb{R}}^{\oplus}\overline{A_f(\xi)}\,\ud\xi\,.
\end{equation}

It can be argued that the fact that the
$\overline{A_f(\xi)}$'s have all the same dense domain in $\mathfrak{h}$ guarantees that the decomposition \eqref{eq:Bfdecomposable} of $B_f$ is unique and hence unambiguous: if one also had $B_f=\int_{\mathbb{R}}^{\oplus}B_f(\xi)\ud\,\xi$ for a map $\xi\mapsto B_f(\xi)$ with $\mathcal{D}(B_f(\xi))=\mathcal{D}(\overline{A_f(\xi)})=\mathcal{D}$, a common dense domain in $\mathfrak{h}$, then necessarily $\overline{A_f(\xi)}=B_f(\xi)$ for almost every $\xi\in\mathbb{R}$.

\begin{remark}
As suggestive as it would be, it is however important to observe that the operator of interest, $\mathscr{H}_f$, is \emph{not} decomposable as $\mathscr{H}_f=\int_\mathbb{R}^{\oplus}A_f(\xi)\,\ud \xi$. Indeed, the analogue of condition (i) in \eqref{ed:defB} would be satisfied, but condition (ii) would not. More precisely, by definition an element $\psi\in\mathcal{D}(\int_\mathbb{R}^{\oplus}A_f(\xi)\,\ud \xi)$ does satisfy the property $\psi(\cdot,\xi)\in\mathcal{D}(A_f(\xi))=C^\infty_c(\mathbb{R}^+)$ for every $\xi$, as is the case for the elements of $\mathcal{D}(\mathscr{H}_f)$, but it also satisfies the property
\begin{equation}\label{eq:Afpsi}
 \begin{split}
 +\infty\;&>\int_\mathcal{\mathbb{R}}\big\|A_f(\xi)\psi(\cdot,\xi)\big\|_{\mathfrak{h}}^2\,\ud\xi \\
 &=\iint_{\mathbb{R}^+\times\mathbb{R}}\Big|\Big(-\frac{\partial^2}{\partial x^2}+\frac{\xi^2}{f^2}+\frac{2ff''-f'^2}{4f^2}\Big)\psi(x,\xi)\Big|^2\ud x\,\ud\xi\,,
 \end{split}
\end{equation}
and \eqref{eq:Afpsi} does not necessarily imply that for every $x$ the function $\psi(x,\cdot)$ is the Fourier transform of a $C^\infty_0(\mathbb{R})$-function as it has to be for an element of $\mathcal{D}(\mathscr{H}_f)$. Condition \eqref{eq:Afpsi} is surely satisfied by other functions besides all those in $\mathcal{D}(\mathscr{H}_f)$. In fact, the same reasoning proves the (proper) inclusion
\begin{equation}\label{eq:BextendsHf}
 B_f\;\supset\; \mathscr{H}_f\,.
\end{equation}
\end{remark}

The operator $B_f$ is not just an extension of $\mathscr{H}_f$, it is a closed symmetric extension.

\begin{proposition}[\cite{Pozzoli_MSc2018}]\label{prop:BsymmetricAndClosed}~
\begin{itemize}
 \item[(i)] $B_f$ is symmetric.
 \item[(ii)] $B_f$ is closed.
\end{itemize}
\end{proposition}

\begin{proof}
 Symmetry is immediately checked by means of \eqref{eq:scalar_products_fibred}, thanks to the symmetry of each $\overline{A_f(\xi)}$. Concerning the closedness, let $(\psi_n)_{n\in\mathbb{N}}$, $\psi$, and $\Psi$ be, respectively, a sequence and two functions in $\mathcal{D}(B_f)$ such that $\psi_n\to\psi$ and $B\psi_n\to\Psi$ in $\cH$ as $n\to +\infty$. Thus,
 \[
  \begin{split}
   \int_{\mathbb{R}}\|\psi_n(\cdot,\xi)-\psi(\cdot,\xi)\|_{\mathfrak{h}}^2\,\ud \xi\;&\xrightarrow[]{\;n\to +\infty\;}0\,, \\
   \int_{\mathbb{R}}\big\|\overline{A_f(\xi)}\,\psi_n(\cdot,\xi)-\Psi(\cdot,\xi)\|_{\mathfrak{h}}^2\,\ud \xi\;&\xrightarrow[]{\;n\to +\infty\;}0\,,
  \end{split}
 \]
 which implies that, \emph{up to extracting a subsequence}, and for \emph{almost} every $\xi$, $\psi_n(\cdot,\xi)\to\psi(\cdot,\xi)$ and $\overline{A_f(\xi)}\,\psi_n(\cdot,\xi)\to\Psi(\cdot,\xi)$ in $\mathfrak{h}$ as $n\to +\infty$. Owing to the closedness of $\overline{A_f(\xi)}$, one must conclude that
 \[
  \psi(\cdot,\xi)\in\mathcal{D}(\overline{A_f(\xi)})\qquad\textrm{and}\qquad \overline{A_f(\xi)}\,\psi(\cdot,\xi)\;=\;\Psi(\cdot,\xi)
 \]
 for almost every $\xi$. Therefore,
\[
 \int_\mathcal{\mathbb{R}}\big\|\overline{A_f(\xi)}\psi(\cdot,\xi)\big\|_{\mathfrak{h}}^2\,\ud\xi\;=\;\|\Psi\|_{\cH}^2\;<\;+\infty.
\]
 Both conditions (i) and (ii) of \eqref{ed:defB} are satisfied, which proves that $\psi\in\mathcal{D}(B_f)$ and $B_f\psi=\Psi$, that is, the closedness of $B$. 
\end{proof}

\subsection{Self-adjointness of the auxiliary fibred operator}\label{sec:selfadjfibredop}~

The convenient feature of the auxiliary operator $B_f$ is the possibility of qualifying its self-adjointess in terms of the same property in each fibre.

One direction of this fact is the following application of the well-known property \cite[Theorem XIII.85(i)]{rs4}:

\begin{proposition}\label{prop:RSpropBA}
 If $A_f(\xi)$ is essentially self-adjoint for each $\xi\in\mathbb{R}$, then $B_f$ is self-adjoint.
\end{proposition}

Let us focus on the opposite direction.

\begin{proposition}[\cite{Pozzoli_MSc2018}]\label{prop:Bselfadj-implies-Axiselfadj}
 If $B_f$ is self-adjoint, then $A_f(\xi)$ is essentially self-adjoint for almost every $\xi\in\mathbb{R}$.
\end{proposition}

\begin{proof}
 It follows by assumption that for any $\varphi\in\cH$ there exists $\psi_\varphi\in\mathcal{D}(B_f)$ with $\varphi=(B_f+\ii)\psi_\varphi$. Thus, as an identity in $\mathfrak{h}$,
 \[
  \varphi(\cdot,\xi)\;=\;\big(\overline{A_f(\xi)}+i\big)\,\psi_\varphi(\cdot,\xi)\qquad\textrm{for almost every }\,\xi\,.
 \]
 In particular, let us run $\varphi$ over all the $C^\infty_c(\mathbb{R}^+_x\times\mathbb{R}_\xi)$-functions and let us fix $\xi_0\in\mathbb{R}$: then obviously $\varphi(\cdot,\xi_0)$ spans the whole space of $C^\infty_c(\mathbb{R}^+_x)$-functions, which is a dense of $\mathfrak{h}$: with this choice the above identity implies that $\mathrm{ran}(\overline{A_f(\xi_0)}+i\mathbbm{1}\big)$ is dense in $\mathfrak{h}$ and hence $A_f(\xi_0)$ is essentially self-adjoint. 
\end{proof}


In turn, the essential self-adjointness of $A_f(\xi)$ can be now studied by means of very classical methods.

\subsection{Weyl's analysis in each fibre}~

Let us re-write 
\begin{equation}
 A_f(\xi)\;=\;-\frac{\ud^2}{\ud x^2}+W_{\xi,f}\,,\qquad W_{\xi,f}(x)\;:=\;\frac{\xi^2}{f^2}+\frac{2ff''-f'^2}{4f^2}\,.
\end{equation}
Owing to assumptions \eqref{eq:assumtpions_f}, $W_{\xi,f}$ is a non-negative continuous function on $\mathbb{R}^+$. With the choice $f(x)=x^{-\alpha}$ it takes the form
\begin{equation}
 W_{\xi,\alpha}\;=\;\xi^2 x^{2\alpha}+\frac{\,\alpha(2+\alpha)\,}{4x^2}\,.
\end{equation}

The essential self-adjointness of $A_f(\xi)$ is controlled by Weyl's limit-point/limit-circle analysis \cite[Sect.~X.1]{rs2}. Thanks to the continuity and non-negativity of $W_{\xi,\alpha}$, $A_f(\xi)$ is always in the limit point \emph{at infinity} -- it suffices to take $M(x)=x^2$ in \cite[Theorem X.8]{rs2} -- so the analysis is boiled down to the sole behaviour \emph{at zero}. Here one has two possibilities:
\begin{itemize}
 \item if $2ff''-f'^2\geqslant 3 x^{-2} f^2$, then $W_{\xi,\alpha}(x)\geqslant\frac{3}{4x^2}$, in which case $A_f(\xi)$ is in the limit point at zero \cite[Theorem X.10]{rs2};
 \item if instead  $2ff''-f'^2\leqslant (3-\varepsilon) x^{-2} f^2$ for some $\varepsilon>0$, since $f^{-2}\leqslant\kappa^{-2}$ around $x=0$, then $W_{\xi,\alpha}(x)\leqslant\kappa^{-2}\xi^2+(3-\varepsilon)x^{-2}$, whence also, for some $\xi$-dependent $\widetilde\varepsilon\in(0,\varepsilon)$, $W_{\xi,\alpha}(x)\leqslant(3-\widetilde{\varepsilon})x^{-2}$: in this case $A_f(\xi)$ is in the limit circle at zero \cite[Theorem X.10]{rs1}.
\end{itemize}

Weyl's criterion \cite[Theorem X.7]{rs1} then leads to the following conclusion.

\begin{proposition}\label{prop:Axiselfadjointness}
 Let $\xi\in\mathbb{R}$ and let $f$ satisfy assumptions \eqref{eq:assumtpions_f}.
 \begin{itemize}
  \item[(i)] If $2ff''-f'^2\geqslant 3 x^{-2} f^2$, then $A_f(\xi)$ is essentially self-adjoint.
  \item[(ii)] If $2ff''-f'^2\leqslant (3-\varepsilon) x^{-2} f^2$ for some $\varepsilon>0$, then $A_f(\xi)$ is not essentially self-adjoint and admits a one-real-parameter family of self-adjoint extensions. 
 \end{itemize}
\end{proposition}

The two alternatives in Proposition \ref{prop:Axiselfadjointness} are not mutually exclusive for generic admissible $f$'s, but they are when $f(x)=x^{-\alpha}$, for in this case
\[
 \frac{2ff''-f'^2}{4f^2}\;=\;\frac{\,\alpha(2+\alpha)\,}{x^2}
\]
and the possibilities are only $0<\alpha<1$ and $\alpha\geqslant 1$. The conclusion is therefore:

\begin{corollary}\label{cor:alpha}
 Let $\xi\in\mathbb{R}$. The operator $A_\alpha(\xi)$ is essentially self-adjoint if and only if $\alpha\geqslant 1$.
\end{corollary}

\section{Proofs of the main results}\label{sec:proofs}

Let us present in this Section the proofs of our main theorems.

\subsection{Geodesic incompleteness}\label{sec:incompleteness}~

The very fact that each $M_\alpha$, when $\alpha>0$, is geodesically incomplete is straightforward, for $M_\alpha$ is obviously incomplete as a metric space (which can be seen by the non-convergent Cauchy sequence of points $(n^{-1},y_0)\in M$ as $n\to\infty$), and the conclusion then follows from a standard Hopf-Rinow theorem 
\cite[Theorem 2.8, Chapter 7]{DoCarmo-Riemannian}.

More interestingly, as is shown now, it is not just possible to find \emph{one} geodesic curve that passes through a given arbitrary point $(x_0,y_0)\in M$ and reaches the boundary $\partial M$ in finite time in the past or in the future -- which is in fact geodesic incompleteness and in the present case is trivially seen by considering the horizontal line $y=y_0$ in the discussion that follows -- but furthermore it can be proved that \emph{all} geodesics passing through $(x_0,y_0)$ at $t=0$ intercept the $y$-axis at finite times $t_\pm$ with $t_-<0<t_+$, with the sole exception of the geodesic line $y=y_0$ along which the boundary is reached only in one direction of time.

Let us recall that, as a consequence of Pontryagin's maximum principle (see, e.g., \cite[Sect.~3.4]{Agrachev-Barilari-Boscain-subriemannian} or \cite[Sect.~2.2]{Boscain-Laurent-2013}), the geodesics on $M_\alpha$ are projections onto $M$ of solutions to the Hamilton equations associated with the Hamiltonian that with respect to the orthonormal frame \eqref{eq:frame} reads
\begin{equation}\label{eq:Hamiltonian_for_geodesics}
 h_\alpha(x,y;P_x,P_y)\;:=\;\frac{1}{2}\big(\langle X_1,P\rangle_{\mathbb{R}^2}^2+\langle X_2^{(\alpha)},P\rangle_{\mathbb{R}^2}^2\big) \;=\;\frac{1}{2}\big(P_x^2+x^{2\alpha} P_y^2\big)\,,
 \end{equation}
where $P:=(P_x,P_y)\in T^*_{(x,y)}M$ is the vector of the momenta associated with the coordinates $(x,y)$. The corresponding Hamiltonian system is therefore
\begin{equation}\label{eq:Hamiltoniansystem}
 \begin{split}
  \dot{x}\;&=\;P_x\,,\qquad\qquad \dot{P}_x\;=\;-\alpha x^{2\alpha-1}\,, \\
  \dot{y}\;&=\;x^{2\alpha} P_y\,,\qquad\;\, \dot{P}_y\;=\;0\,.
 \end{split}
\end{equation}
The local existence and uniqueness of a solution to \eqref{eq:Hamiltoniansystem} with prescribed values of $(x,y)$ and $(P_x,P_y)$ at $t=0$ is standard.

\begin{figure}
\begin{center}
\includegraphics[width=5cm]{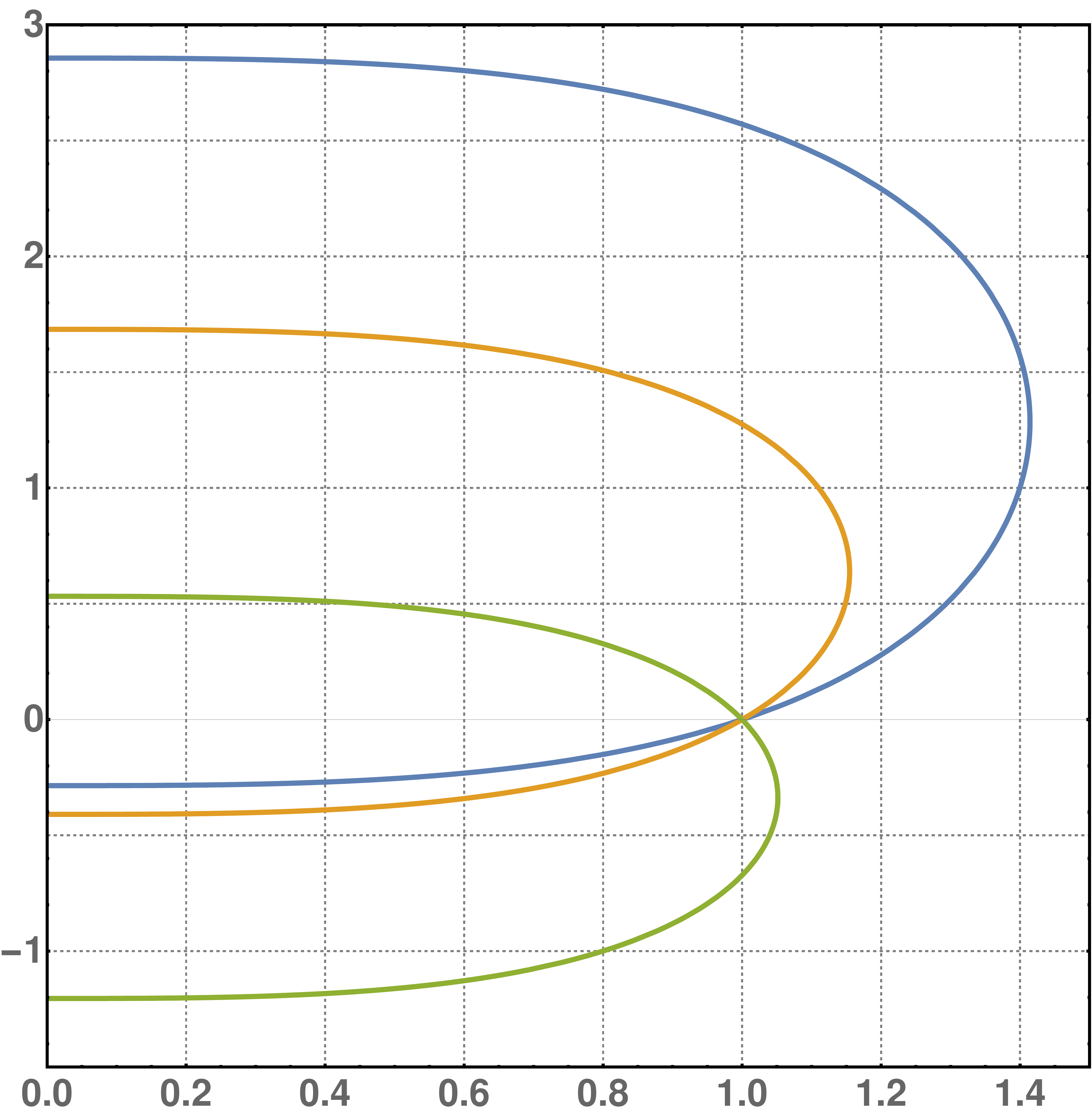}\quad\includegraphics[width=5cm]{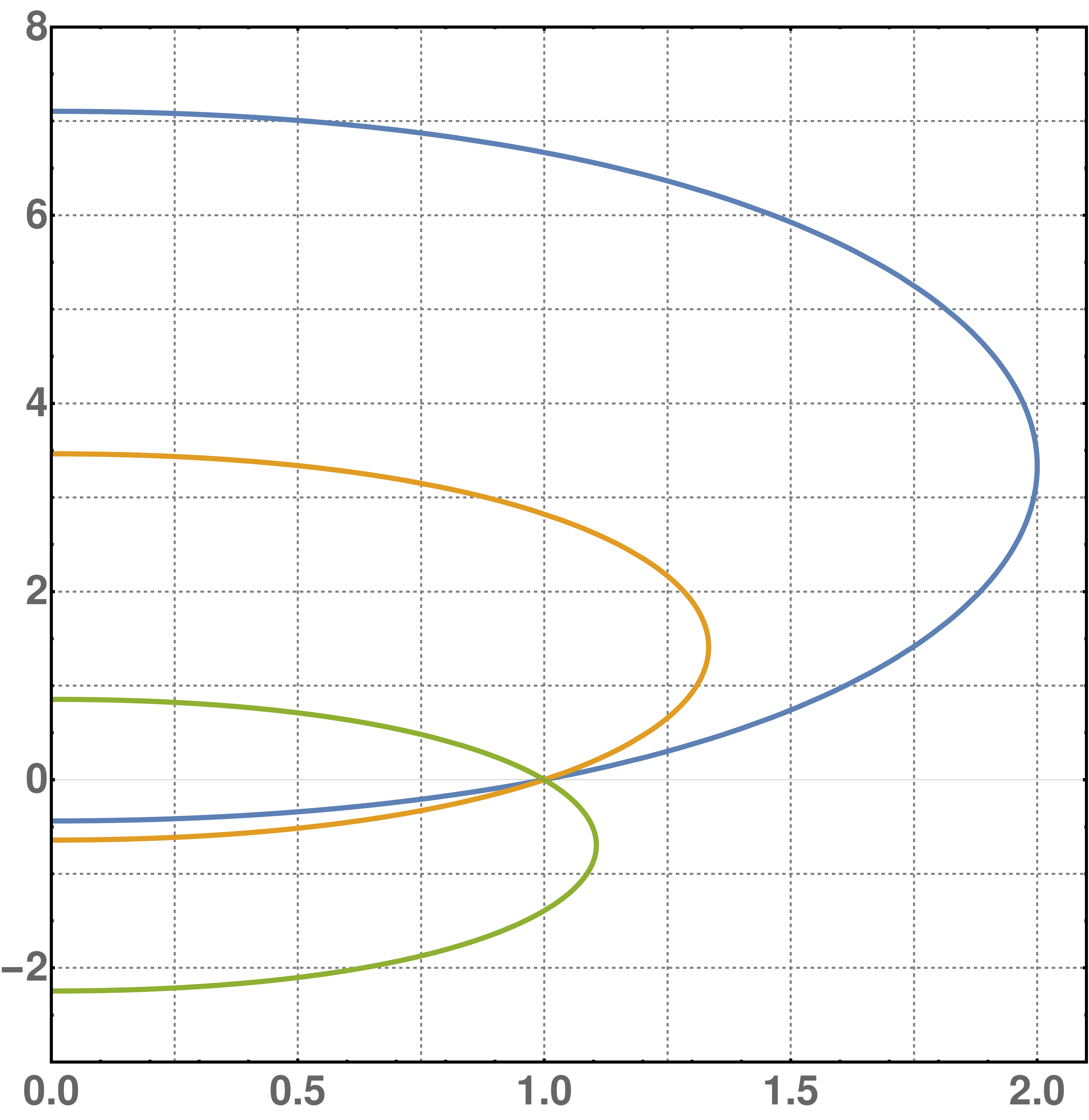}
\caption{Geodesics obtained from \eqref{eq:system_for_geodesics} for various choices of $\theta$. Left: $\alpha=\frac{1}{2}$. Right: $\alpha=1$.} \label{fig:geodesics}
\end{center}
\end{figure}

One deduces from \eqref{eq:Hamiltoniansystem} that the geodesic passing at $t=0$ through the point $(x(0),y(0))=(1,0)\in M$ with direction $(\dot{x}(0),\dot{y}(0))=(\cos\theta,\sin\theta)$, for fixed $\theta\in[0,2\pi)$, is the solution $(x(t),y(t))$ to
\begin{equation}\label{eq:system_for_geodesics}
 \begin{cases}
  \;\ddot{x}+\alpha\,\sin^2\theta\,x^{2\alpha-1}\,=\,0\,,\quad & x(0)=1\,,\quad \dot{x}(0)=\cos\theta \,, \\
  \;y(t)\,=\,\sin\theta\int_0^t x(\tau)^{2\alpha}\,\ud \tau\,.
 \end{cases}
\end{equation}

The exceptional cases $\theta=0$ and $\theta=\pi$ yield, respectively, the geodesics $(1+t,0)$ and $(1-t,0)$, both reaching the boundary $\partial M$, respectively at the instants $t=-1$ and $t=1$.

Generically, there are two instants $t_\pm$ with $t_-<0<t_+$ such that $x(t_\pm)=0$ (Fig.~\ref{fig:geodesics}).

This is seen as customary by exploiting the conservation of $h_\alpha$ along each geodesic, that is, the conservation of the quantity
\begin{equation}\label{eq:Exxdot}
E(x,\dot x) \; := \; \frac{1}{2}\big(\dot{x}^2+ \sin^2 \theta \,x^{2 \alpha}\big)\;=\;\frac{1}{2}
\end{equation}
computed from \eqref{eq:Hamiltonian_for_geodesics}-\eqref{eq:Hamiltoniansystem} with the initial values $(x(0),y(0))=(1,0)$, $(\dot{x}(0),\dot{y}(0))=(\cos\theta,\sin\theta)$, or equivalently computed from \eqref{eq:system_for_geodesics}. Suitably interpreting the sign of $\dot{x}$ and integrating the identity
\begin{equation}
 \ud t\;=\;\pm \big(2 E(x(0),\dot{x}(0))-\sin^2 \theta\, x^{2\alpha}\big)^{-\frac{1}{2}}\ud x
\end{equation}
obtained from \eqref{eq:Exxdot}, one computes the time $T(x_{\mathrm{in}}\to x_{\mathrm{fin}})$ needed for $x(t)$ to reach a final value $x_{\mathrm{fin}}$ from and initial value $x_{\mathrm{in}}$ along a geodesic $\gamma$.
Thus,
\begin{itemize}
 \item if $\cos\theta\leqslant 0$, then $\dot{x}(t)=\cos^2\theta-\alpha\sin^2\int_0^t x(\tau)^{2\alpha-1}\ud\tau<0$, and therefore
 \[
  \begin{split}
   T\;&=\;-\int_{1}^0\frac{ \ud x}{\sqrt{1-\sin^2 \theta \, x^{2\alpha}}} \;\leqslant\; \int_0^1 \frac{\ud x}{\sqrt{1-x^{2\alpha}}} \\
   &=\;\frac{1}{\alpha}\int_0^1  \frac{ \ud s}{\,s^{1-\frac{1}{\alpha}} \sqrt{1- s}\sqrt{1+s}}\;<\;+\infty\,;
  \end{split}
 \]
 \item if $\cos \theta > 0$ and $\sin \theta \neq 0$, then $\dot{x}$ changes sign at the critical point $x=x_c:=|\sin\theta|^{-\frac{1}{\alpha}}$ and the above computation gets modified as
\begin{equation*}
\begin{split}
T\;&=\; \int_1^{x_c} \frac{\ud x}{\sqrt{1-\sin^2 \theta \,x^{2 \alpha}}} - \int_{x_c}^0 \frac{\ud x}{\sqrt{1-\sin^2 \theta \,x^{2 \alpha}}}\\
&= \;\int_0^{|\sin\theta|^{-\frac{1}{\alpha}}}\!\!\!\frac{\ud x}{\sqrt{1-\sin^2\theta\,x^{2\alpha}}\,}\;=\; \frac{1}{|\sin \theta|^{\frac{1}{\alpha}}} \int_0^1 \frac{\ud s}{\sqrt{1-s^{2\alpha}}} \;<\;+\infty\,.
\end{split}
\end{equation*}
 \end{itemize}

This argument shows the finiteness of the above-mentioned positive instant $t_+$ of reach of $\partial M$, and the finiteness of $t_-$ follows by the same argument reverting the sign of $t$ in the equations.

Clearly the choice of the initial point $(1,0)$ is non-restrictive and \emph{any} other point $(x_0,y_0)\in M$ can be treated the same way, thanks to the translational invariance of the metric along the $y$-direction.

\subsection{Absence of geometric confinement}\label{sec:no_confinement}~

We already argued in Section \ref{sec:unitarily_equiv} that it is equivalent to study the essential self-adjointness in $\cH=L^2(\mathbb{R}^+\times\mathbb{R},\ud x\,\ud\xi)$ of the operator $\mathscr{H}_f$ defined in \eqref{eq:Hscrf}.

Let us work here in the regime $2ff''-f'^2\leqslant (3-\varepsilon) x^{-2} f^2$ for some $\varepsilon>0$, or in particular, when $f(x)=x^{-\alpha}$, the regime $\alpha\in(0,1)$.

Proposition \ref{prop:Axiselfadjointness}(ii) (in particular, Corollary \ref{cor:alpha}) then show that for no $\xi\in\mathbb{R}$ can $A_f(\xi)$ be essentially self-adjoint. Owing to Proposition \ref{prop:Bselfadj-implies-Axiselfadj}, the auxiliary operator $B_f$ defined in \eqref{ed:defB} is not self-adjoint.

On the other hand, $B_f$ is a closed symmetric extension of $\mathscr{H}_f$, owing to \eqref{eq:BextendsHf} and to Proposition \ref{prop:BsymmetricAndClosed}, whence $\overline{\mathscr{H}_f}\subset B_f$.

Now, if $\mathscr{H}_f$ was essentially self-adjoint, it could not be $\overline{\mathscr{H}_f}= B_f$, because this would violate the lack of self-adjointness of $B_f$. But it could not happen either that $B_f$ is a \emph{proper} extension of $\overline{\mathscr{H}_f}$, because self-adjoint operators are maximally symmetric.

Therefore, $\mathscr{H}_f$ is not essentially self-adjoint. In this regime the Grushin plane does not provide geometric quantum confinement.

Theorems \ref{thm:no_confinement} and \ref{thm:generalisation_to_f}(ii) are thus proved.

\subsection{Presence of geometric confinement}\label{sec:confimenent}~

Let us work now in the regime $2ff''-f'^2\geqslant 3 x^{-2} f^2$, or in particular, when $f(x)=x^{-\alpha}$, the regime $\alpha\in[1,+\infty)$.

Proposition \ref{prop:Axiselfadjointness}(i) (in particular, Corollary \ref{cor:alpha}) then shows that for all $\xi\in\mathbb{R}$ the operator $A_f(\xi)$ is essentially self-adjoint, and therefore, owing to Proposition \ref{prop:RSpropBA}, the auxiliary operator $B_f$ is self-adjoint.

Let us now argue that in the present regime one has
\begin{equation}\label{eq:Hfstar_in_Bf}
 \mathscr{H}_f^*\;\subset\;B_f\,.
\end{equation}
For \eqref{eq:Hfstar_in_Bf} it is sufficient to prove that $\mathcal{D}(\mathscr{H}_f^*)\subset\mathcal{D}(B_f)$, for the differential action of the two operators is the same, as is evident from Lemma \ref{lemma:Hfstar}.

For generic $F\in\mathcal{D}(\mathscr{H}_f^*)$ formula \eqref{eq:Hfstar} gives
\[\tag{*}
\begin{split}
 +\infty\;&>\;\Big\|\Big(-\frac{\partial^2}{\partial x^2}+\frac{\xi^2}{f^2}+\frac{2ff''-f'^2}{4f^2}\Big)F\Big\|^2_{\cH} \\
 &=\;\int_{\mathbb{R}}\ud\xi\, \Big\|\Big(-\frac{\ud^2}{\ud x^2}+\frac{\xi^2}{f^2}+\frac{2ff''-f'^2}{4f^2}\Big)F(\cdot,\xi)\Big\|^2_{\mathfrak{h}} \,,
\end{split}
\]
whence
\[
 \Big\|\Big(-\frac{\ud^2}{\ud x^2}+\frac{\xi^2}{f^2}+\frac{2ff''-f'^2}{4f^2}\Big)F(\cdot,\xi)\Big\|^2_{\mathfrak{h}}\;<\;+\infty\quad\textrm{for almost every }\xi\in\mathbb{R}\,.
\]
The latter formula, owing to \eqref{eq:Afstar}, can be re-written as
\[
 F(\cdot,\xi)\,\in\,\mathcal{D}(A_f(\xi)^*)\quad\textrm{for almost every }\xi\in\mathbb{R}
\]
and since in the present regime $A_f(\xi)^*=\overline{A_f(\xi)}$, we can also write
\[\tag{**}
 F(\cdot,\xi)\,\in\,\mathcal{D}(\overline{A_f(\xi)})\quad\textrm{for almost every }\xi\in\mathbb{R}\,.
\]
Now, (*) and (**) imply that $F\in\mathcal{D}(B_f)$, thus establishing the property \eqref{eq:Hfstar_in_Bf}.

To complete the argument, let us combine the inclusion $B_f\supset \overline{\mathscr{H}_f}$, that follows from \eqref{eq:BextendsHf} and from the closedness of $B_f$, with the inclusion $B_f\subset \overline{\mathscr{H}_f}$, that follows from \eqref{eq:Hfstar_in_Bf} by taking the adjoint, because $\overline{\mathscr{H}_f}=\mathscr{H}_f^{**}\supset B_f^*=B_f$, having used the self-adjointness of $B_f$ valid in the present regime. Since then $\overline{\mathscr{H}_f}=B_f$, the conclusion is that $\mathscr{H}_f$ is essentially self-adjoint.

In this regime there is geometric quantum confinement in the Grushin plane. Theorems \ref{thm:no_confinement} and \ref{thm:generalisation_to_f}(i) are thus proved.


\subsection{Infinite deficiency index}\label{sec:infinite_def_index}~

When $\mathscr{H}_f$ is \emph{not} self-adjoint, necessarily the spaces $\ker(\mathscr{H}_f^*\pm\ii\mathbbm{1})$ are non-trivial. Let us show now that in this case
\begin{equation}\label{eq:infinite_def_space}
 \dim\ker(\mathscr{H}_f^*+\ii\mathbbm{1})\;=\;\dim\ker(\mathscr{H}_f^*-\ii\mathbbm{1})\;=\;\infty\,,
\end{equation}
thus proving Theorem \ref{thm:generalisation_to_f}(iii).

First, since by assumption each $\overline{A_f(\xi)}$ is not self-adjoint, there exists $\varphi_{\xi}\in\mathcal{D}(A_f(\xi)^*)$ with $\|\varphi_{\xi}\|_{\mathfrak{h}}=1$ such that
\begin{equation}\label{eq:eigenf_adjoint}
 A_f(\xi)^*\,\varphi_{\xi}\;=\;\ii \varphi_{\xi}\,.
\end{equation}
From this we shall now deduce that for any compact interval $J\subset\mathbb{R}$, with $\mathbf{1}_J$ characteristic function of $J$, the function $\Phi_J$ defined by $\Phi_J(x,\xi):=\varphi_{\xi}(x)\mathbf{1}_J(\xi)$ satisfies
\begin{equation}\label{eq:Phi_eigenfunction}
 \Phi_J\,\in\,\mathcal{D}(\mathscr{H}_f^*)\,,\qquad \mathscr{H}_f^*\,\Phi_J\;=\;\ii\,\Phi_J\,.
\end{equation}

The fact that $\Phi_J\in\cH$ follows from $\|\Phi_J\|_\cH^2=\int_\mathbb{R}\ud\xi\,\mathbf{1}_J(\xi)\,\|\varphi_{\xi}\|_{\mathfrak{h}}^2=|J|$, where $|J|$ denotes the Lebesgue measure of $J$. Moreover, for any $\psi\in\mathcal{D}(\mathscr{H}_f)$,
\[
 \begin{split}
  \langle\Phi_J,\mathscr{H}_f \psi\rangle_{\cH}\;&=\;\iint_{\mathbb{R}^+\times\mathbb{R}}\ud x\,\ud\xi\,\overline{\varphi_{\xi}(x)}\mathbf{1}_J(\xi)\,A_f(\xi)\psi(x,\xi) \\
  &=\int_J\ud\xi\,\langle \varphi_\xi,A_f(\xi)\psi(\cdot,\xi)\rangle_{\mathfrak{h}}\;=\;\int_J\ud\xi\,\langle A_f(\xi)^*\varphi_\xi,\psi(\cdot,\xi)\rangle_{\mathfrak{h}} \\
  &=\;-\ii\int_J\ud\xi\,\langle \varphi_\xi,\psi(\cdot,\xi)\rangle_{\mathfrak{h}}\;=\;\langle\,\ii\,\Phi_J, \psi\rangle_{\cH}
 \end{split}
\]
where we used \eqref{eq:eigenf_adjoint} in the fourth identity, and this establishes precisely \eqref{eq:Phi_eigenfunction}.

By the arbitrariness of $J$, and the obvious orthogonality $\Phi_J\perp\Phi_{J'}$ whenever $J\cap J'=\varnothing$, we have thus produced an 
infinity of linearly independent eigenfunctions of $\mathscr{H}_f^*$ with eigenvalue $\ii$, and the same can be clearly repeated for the eigenvalue $-\ii$. This completes the proof of \eqref{eq:infinite_def_space}.

\subsection{Comparison with the compact case}\label{sec:additional_remarks}~


We already mentioned that in \cite[Sect.~3.2]{Boscain-Laurent-2013} and in \cite{Boscain-Prandi-JDE-2016} the closely related problem of geometric quantum confinement was solved in the manifold \eqref{eq:MtildeBoscain} -- that we can generalise here to
\begin{equation}
 \widetilde{M}_f\:\equiv\:(\mathbb{R}^+_x\times\mathbb{T}_y,g_f)\,,
\end{equation}
with the usual $g_f$ from \eqref{eq:Mf} and $f$ satisfying assumptions \eqref{eq:assumtpions_f}.

The compactness of $\mathbb{T}$ trivialises the constant-fibre direct integral scheme: the Fourier transform in the $y$ variable naturally makes the conjugate space an infinite orthogonal direct sum of single-Fourier-mode Hilbert spaces, and our \eqref{eq:Hscrf} gets simplified to
\begin{equation}\label{eq:directorthsumcomp}
   \mathscr{H}_f\;=\;\bigoplus_{k\in\mathbb{Z}}\mathscr{H}_f^{(k)}
\end{equation}
with operators
\begin{equation}\label{eq:fibreobcomp}
 \begin{split}
 \mathscr{H}_f^{(k)}\;&=\;-\frac{\ud^2}{\ud x^2}+\frac{k^2}{f^2}+\frac{2ff''-f'^2}{4f^2}
 \end{split}
\end{equation}
in the fibre Hilbert space $\mathfrak{h}=L^2(\mathbb{R}^+,\ud x)$. The continuous variable $\xi$ is thus replaced by the discrete variable $k$.

In this case the study we made in Section \ref{sec:selfadjfibredop} is not needed: indeed, it is a standard exercise that the essential self-adjointness of $\mathscr{H}_f$ is tantamount as the essential self-adjointness of all the $\mathscr{H}_f^{(k)}$'s, and the latter is fully controlled by Weyl's analysis.

It is worth remarking a noticeable difference between the compact and the non-compact case as far as the essential self-adjointness of the minimally defined Laplace-Beltrami operator is concerned, which emerges when there is \emph{no singularity} in the metric $g_f$ at $x=0$ -- for concreteness, when $f(x)=x^{-\alpha}$ with $\alpha<0$.

In the `Grushin cylinder' $M_\alpha=(\mathbb{R}^+_x\times\mathbb{T}_y,g_\alpha)$, as recently determined in \cite[Theorem 1.6]{Boscain-Prandi-JDE-2016}, when one considers generic $\alpha\in\mathbb{R}$ it turns out that
\begin{itemize}
 \item essential self-adjointness holds for $\alpha\in(-\infty,-3]\cup[1,+\infty)$ -- this is seen by studying in the usual way each fibre operator $ \mathscr{H}_\alpha^{(k)}$ (the analogue of \eqref{eq:fibreobcomp}) and then taking the (analogue of the) infinite orthogonal sum \eqref{eq:directorthsumcomp};
 \item in particular, the lack of essential self-adjointness for $\alpha\in(-3,1)$ is due to the Fourier mode $k=0$ only, when $\alpha\in(-3,-1]$, and is due instead to \emph{all} Fourier modes $k\in\mathbb{Z}$ when $\alpha\in(-1,1)$, that is, in the latter case all  $\mathscr{H}_\alpha^{(k)}$'s fail to be essentially self-adjoint on $\mathfrak{h}=L^2(\mathbb{R}^+,\ud x)$.
\end{itemize}

As opposite to that, we can study the same problem in the Grushin plane $M_\alpha=(\mathbb{R}^+_x\times\mathbb{R}_y,g_\alpha)$ also when $\alpha<0$ by virtually repeating almost verbatim the analysis of Sections \ref{sec:preliminaries}, \ref{sec:no_confinement}, and \ref{sec:confimenent}. Concerning the fibre operator $A_\alpha(\xi)$ on $\mathfrak{h}=L^2(\mathbb{R}^+,\ud x)$ we can find that
\begin{equation}\label{eq:summaryhere}
 \begin{split}
  \textrm{if }\alpha \in (-\infty,-1) \cup [1,+\infty), & \quad  \textrm{$ A_\alpha(\xi)$ is ess.~self-adj.~for almost every $\xi \in \mathbb{R}$,} \\
  \textrm{if } \alpha = -1, & \quad  \textrm{$ A_\alpha(\xi)$ is ess.~self-adj.~for $|\xi| \geq 1$,} \\
  \textrm{if } \alpha \in (-1,1), & \quad  \textrm{$ A_\alpha(\xi)$ is not ess.~self-adj.}
 \end{split}
\end{equation}
From \eqref{eq:summaryhere}, taking the direct integral of the $A_\alpha(\xi)$'s, we conclude that
\begin{itemize}
 \item essential self-adjointness holds for $\alpha\in(-\infty,-1)\cup[1,+\infty)$;
 \item the lack of quantum confinement in the complement range $\alpha\in[-1,1)$ is due to a `transmission' through the boundary only by the Fourier modes $\xi\in(-1,1)$ if $\alpha=-1$, and to a transmission by \emph{all} Fourier modes $\xi\in\mathbb{R}$ if $\alpha\in(-1,1)$.
\end{itemize}

This shows a difference between the compact and the non-compact case both in the regimes of essential self-adjointness (when $\alpha \in (-3,-1)$ geometric quantum confinement holds in the Grushin plane and not in the Grushin cylinder) and in the Fourier modes responsible for the transmission.

\section*{Acknowledgements}

We warmly thank U.~Boscain for bringing this problem and the related literature to our attention and for several enlightening discussions on the subject. We also thank for the kind hospitality the Istituto Nazionale di Alta Matematica (INdAM), Rome, where part of this work was carried on. This project is also partially funded by the European Union's Horizon 2020 Research and Innovation Programme under the Marie Sklodowska-Curie grant agreement no.~765267.

%

\begin{thebibliography}{10}

\bibitem{Agrachev-Barilari-Boscain-subriemannian}
{\sc A.~Agrachev, D.~Barilari, and U.~Boscain}, {\em {Introduction to
  Riemannian and Sub-Riemannian geometry from Hamiltonian viewpoint}}, SISSA
  preprint 09/2012/M (2012).

\bibitem{Agrachev-Boscain-Sigalotti-2008}
{\sc A.~Agrachev, U.~Boscain, and M.~Sigalotti}, {\em {A {G}auss-{B}onnet-like
  formula on two-dimensional almost-{R}iemannian manifolds}}, Discrete Contin.
  Dyn. Syst., 20 (2008), pp.~801--822.

\bibitem{Boscain-Laurent-2013}
{\sc U.~Boscain and C.~Laurent}, {\em {The {L}aplace-{B}eltrami operator in
  almost-{R}iemannian geometry}}, Ann. Inst. Fourier (Grenoble), 63 (2013),
  pp.~1739--1770.

\bibitem{Boscain-Prandi-JDE-2016}
{\sc U.~Boscain and D.~Prandi}, {\em {Self-adjoint extensions and stochastic
  completeness of the {L}aplace-{B}eltrami operator on conic and anticonic
  surfaces}}, J. Differential Equations, 260 (2016), pp.~3234--3269.

\bibitem{Braverman-Milatovich-Shubin-2002}
{\sc M.~Braverman, O.~Milatovich, and M.~Shubin}, {\em {Essential
  selfadjointness of {S}chr{\"o}dinger-type operators on manifolds}}, Uspekhi
  Mat. Nauk, 57 (2002), pp.~3--58.

\bibitem{Calin-Chang-SubRiemannianGeometry}
{\sc O.~Calin and D.-C. Chang}, {\em {Sub-{R}iemannian geometry}}, vol.~126 of
  {Encyclopedia of Mathematics and its Applications}, Cambridge University
  Press, Cambridge, 2009.
\newblock General theory and examples.

\bibitem{Cycon-F-K-S-Schroedinger_ops}
{\sc H.~L. Cycon, R.~G. Froese, W.~Kirsch, and B.~Simon}, {\em {Schr{\"o}dinger
  operators with application to quantum mechanics and global geometry}}, {Texts
  and Monographs in Physics}, Springer-Verlag, Berlin, study~ed., 1987.

\bibitem{DoCarmo-Riemannian}
{\sc M.~P. do~Carmo}, {\em {Riemannian geometry}}, {Mathematics: Theory \&
  Applications}, Birkh{\"a}user Boston, Inc., Boston, MA, 1992.
\newblock Translated from the second Portuguese edition by Francis Flaherty.

\bibitem{Franceschi-Prandi-Rizzi-2017}
{\sc V.~Franceschi, D.~Prandi, and L.~Rizzi}, {\em {On the essential
  self-adjointness of singular sub-Laplacians}}, arXiv:1708.09626 (2017).

\bibitem{Gaffney-1954}
{\sc M.~P. Gaffney}, {\em {A special {S}tokes's theorem for complete
  {R}iemannian manifolds}}, Ann. of Math. (2), 60 (1954), pp.~140--145.

\bibitem{Hall-2013_QuantumTheoryMathematicians}
{\sc B.~C. Hall}, {\em {Quantum theory for mathematicians}}, vol.~267 of
  {Graduate Texts in Mathematics}, Springer, New York, 2013.

\bibitem{Nenciu-Nenciu-2009}
{\sc G.~Nenciu and I.~Nenciu}, {\em {On confining potentials and essential
  self-adjointness for {S}chr{\"o}dinger operators on bounded domains in {$\Bbb
  R^n$}}}, Ann. Henri Poincar{\'e}, 10 (2009), pp.~377--394.

\bibitem{Pozzoli_MSc2018}
{\sc E.~Pozzoli}, {\em {Models of quantum confinement and perturbative methods
  for point interactions}}, Master Thesis (2018).

\bibitem{Prandi-Rizzi-Seri-2016}
{\sc D.~Prandi, L.~Rizzi, and M.~Seri}, {\em {Quantum confinement of
  non-complete Riemannian manifolds}}, arXiv:1609.01724 (2016).

\bibitem{rs1}
{\sc M.~Reed and B.~Simon}, {\em {Methods of {M}odern {M}athematical
  {P}hysics}}, vol.~1, New York Academic Press, 1972.

\bibitem{rs2}
\leavevmode\vrule height 2pt depth -1.6pt width 23pt, {\em {Methods of modern
  mathematical physics. {II}. {F}ourier analysis, self-adjointness}}, Academic
  Press [Harcourt Brace Jovanovich, Publishers], New York-London, 1975.

\bibitem{rs4}
\leavevmode\vrule height 2pt depth -1.6pt width 23pt, {\em {Methods of modern
  mathematical physics. {IV}. {A}nalysis of operators}}, Academic Press
  [Harcourt Brace Jovanovich, Publishers], New York-London, 1978.

\bibitem{schmu_unbdd_sa}
{\sc K.~Schm{\"u}dgen}, {\em {Unbounded self-adjoint operators on {H}ilbert
  space}}, vol.~265 of {Graduate Texts in Mathematics}, Springer, Dordrecht,
  2012.

\bibitem{Ward-2016}
{\sc A.~D. Ward}, {\em {The essential self-adjointness of {S}chr{\"o}dinger
  operators on domains with non-empty boundary}}, Manuscripta Math., 150
  (2016), pp.~357--370.

\end{thebibliography}

\def\cprime{$'$}

\end{document}